\documentclass[english]{amsart}
\usepackage[latin9]{inputenc}
\usepackage{amsthm}
\usepackage{amssymb}
\usepackage{esint}

\makeatletter

\newcommand{\noun}[1]{\textsc{#1}}

\numberwithin{equation}{section}
\numberwithin{figure}{section}
  \theoremstyle{remark}
  \newtheorem*{rem*}{\protect\remarkname}
\theoremstyle{plain}
\newtheorem{thm}{\protect\theoremname}
  \theoremstyle{definition}
  \newtheorem{defn}[thm]{\protect\definitionname}
  \theoremstyle{remark}
  \newtheorem{rem}[thm]{\protect\remarkname}
  \theoremstyle{plain}
  \newtheorem{lem}[thm]{\protect\lemmaname}
  \theoremstyle{plain}
  \newtheorem*{thm*}{\protect\theoremname}
  \theoremstyle{plain}
  \newtheorem{cor}[thm]{\protect\corollaryname}

\makeatother

\usepackage{babel}
  \providecommand{\corollaryname}{Corollary}
  \providecommand{\definitionname}{Definition}
  \providecommand{\lemmaname}{Lemma}
  \providecommand{\remarkname}{Remark}
  \providecommand{\theoremname}{Theorem}
\providecommand{\theoremname}{Theorem}

\begin{document}

\title[Parabolic Equations of Second Order with Critical Drift]{Non-divergence Parabolic Equations of Second Order with Critical Drift in Morrey Spaces}

\date{05/27/15}

\author{Gong Chen}


\email{gc@math.uchicago.edu}

\urladdr{http://www.math.uchicago.edu/\textasciitilde{}gc/}

\address{Department of Mathematics, The University of Chicago, 5734 South
University Avenue, Chicago, IL 60615, U.S.A}
\begin{abstract}
We consider uniformly parabolic equations and inequalities of second
order in the non-divergence form with drift 
\[-u_{t}+Lu=-u_{t}+\sum_{ij}a_{ij}D_{ij}u+\sum b_{i}D_{i}u=0\,(\geq0,\,\leq0)\] in some domain $\Omega\subset \mathbb{R}^{n+1}$.
 We prove a variant of Aleksandrov-Bakelman-Pucci-Krylov-Tso estimate with $L^{p}$ norm of the inhomogeneous term for some number $p<n+1$. Based
 on it, we derive the growth theorems and the interior Harnack inequality.  In this paper, we will only assume
 the drift $b$ is in certain Morrey spaces defined below which are critical under the parabolic scaling but not necessarily to be bounded. This is a continuation of the work in \cite{GC}.
\end{abstract}
\maketitle

\section{Introduction}

\subsection{General Introduction}

The qualitative properties of solutions to partial diff{}erential
equations have been intensively studied for a long time. Following
\cite{GC}, in this note, we continue our discussion on the qualitative
properties of solutions to the uniform parabolic equation of non-divergence
form with drift,
\begin{equation}
-u_{t}+Lu:=-u_{t}+\sum_{ij}a_{ij}D_{ij}u+\sum_{i}b_{i}D_{i}u=0\label{eq:01}
\end{equation}
and the associated inequalities: $-u_{t}+Lu\geq 0$ and $-u_{t}+Lu\leq 0$. Throughout the paper, we use the notations $D_{i}:=\frac{\partial}{\partial x_{i}},\, D_{ij}:=\frac{\partial^{2}}{\partial x_{i}\partial x_{j}}$
and $u_{t}:=\frac{\partial u}{\partial t}$. We assume $b=\left(b_{1},\ldots,b_{n}\right)$ and $a_{ij}$'s are real  measurable, $a_{ij}$'s also satisfy the {\em uniform parabolicity condition}
\begin{equation}
\forall\xi\in\mathbb{R}^{n},\,\nu^{-1}\left|\xi\right|^{2}\leq\sum_{i,j=1}^{n}a_{ij}(X)\xi_{i}\xi_{j},\,\,\,\,\,\,\,\sum_{i,j=1}^{n}a^{2}_{ij}\leq \nu^{2{}}\label{eq:02}
\end{equation}
with some constant $\nu\geq1$, 
$\forall X=(x,t)$ in the domain of definition $\Omega \subset \mathbb{R}^{n+1}$. 

For the drift $b$, we will only require it is in certain
Morrey spaces which are critical under the parabolic scaling. To formulate our setting more precisely, we define Morrey
spaces as following: given some constants $p,q\geq1$ and $\alpha\geq0$ satisfying
 \begin{equation}
\frac{n}{p}+\frac{2}{q}-\alpha=1,\label{eq:04}
\end{equation} 
on the domain of definition $\Omega$, we define
\begin{equation}
M_{p,q}^{\alpha}(\Omega):=\left\{ f\in L_{x}^{p}L_{t}^{q}(\Omega);\,||f||_{M_{p,q}^{\alpha}(\Omega)}:=\sup_{Q_{r}\subset\Omega,\,r>0}r^{-\alpha}||f||_{L_{x}^{p}L_{t}^{q}(Q_{r})}<\infty\right\} \label{eq:24}
\end{equation}
where  $Q_{r}$ is the standard parabolic cylinder defined in Definition \ref{cylinder} and
\[
\left\Vert f\right\Vert _{L_{x}^{p}L_{t}^{q}}:=\left(\intop\left[\int\left|f(x,t)\right|^{q}dt\right]^{\frac{p}{q}}dx\right)^{\frac{1}{p}}.
\]
We will focus on a particular case with $p=q=\frac{1}{\alpha}=n+1$, 
\begin{equation}
b\in M_{n+1,n+1}^{\frac{1}{n+1}}(\Omega)
\end{equation}
with 
\begin{equation}
\sup_{Q_{r}\subset\Omega,\,r>0}\left[\frac{1}{r}\intop_{Q_{r}}|b|^{n+1}dxdt\right]=: S(\Omega)<\infty.\label{eq:23}
\end{equation}
By "critical", we mean that with the $M_{n+1}^{\frac{1}{n+1}}(\Omega)$
norm, the drift is scaling invariant under the parabolic scaling: for
$k>0$,
\[
x\rightarrow k^{-1}x,\,\, t\rightarrow k^{-2}t.
\]
Indeed, suppose $u$ satisfies
\[
-u_{t}+\sum_{ij}a_{ij}D_{ij}u+\sum_{i}b_{i}D_{i}u=0.
\] in a domain $Q\in \mathbb{R}^{n+1}$. Then for any constant $k>0$, let 
\[\tilde{x}=k^{-1}x,\,\tilde{t}=k^{-2}t.\]
Then  $\widetilde{u}\left(\tilde{x},\tilde{t}\right)=u\left(r\tilde{x},r^{2}\tilde{t}\right)$
satisfies the equation
\[
-\widetilde{u}_{\tilde{t}}+\sum_{ij}\widetilde{a_{ij}}D_{ij}\widetilde{u}+\sum_{i}\tilde{b_{i}}D_{i}\widetilde{u}=0,
\]
in $Q_{k}:=\{(x,t),\,(kx,k^{2}t)\in Q \}$.
Note that $\tilde{b}=kb$, so
\[
S^{\frac{1}{n+1}}(\Omega_{k})=\left\Vert \tilde{b}\right\Vert _{M_{n+1,n+1}^{\frac{1}{n+1}}(\Omega_{k})}=\left\Vert b\right\Vert _{M_{n+1,n+1}^{\frac{1}{n+1}}(\Omega)}=S^{\frac{1}{n+1}}(\Omega).
\]

In general, regarding the
scaling, intuitively, there is a competition between the transport term
and the diffusion part. One might expect that for the supercritical scaling
case,  $\frac{n}{p}+\frac{2}{q}-\alpha>1$: the solutions of the equations
have discontinuities \cite{SVZ}. For the critical  situation we
are considering here, we have H\"older continuous solutions, see Theorem \ref{hoelder}. Finally,
if the drift is subcritical with respect to the scaling, i.e. $\frac{n}{p}+\frac{2}{q}-\alpha<1$,
we expect the solutions will be smooth. We should notice $L_{x,t}^{n+1}$ ($p=q=n+1,\,\alpha=0$)
is supercritical with respect to the parabolic scaling. We will discuss a concrete example in the appendix.

We will concentrate on the growth theorems and the interior Harnack inequality for parabolic equations in non-divergence form with critical drift. In order to derive them, we prove a variant of Aleksandrov-Bakelman-Pucci-Krylov-Tso estimate, Theorem \ref{variant}. This variant of Aleksandrov-Bakelman-Pucci-Krylov-Tso
estimate enable us to estimate the supremum of a solution to $-u_{t}+Lu=f$ in a bounded Lipschitz domain in $\mathbb{R}^{n+1}$  in terms of the Dirichlet data on the boundary and the $L^{p}$ norm of $f$ with some constant $p<n+1$ depending on $n,\,\nu,\,S$.

With the assumptions and preparations above, the main results in this paper are then expressed by Theorems \ref{variant} and \ref{harnack}.

\begin{defn}\label{space}
	Given $p\geq1$, for any open set $\Omega\subset\mathbb{R}^{n+1}$, we define the space
	\begin{equation}
	W_{p}(\Omega):= C(\Omega)\cap W_{p,p}^{2,1}(\Omega),\label{eq:W}
	\end{equation} 
	where $f\in W_{p,p}^{2,1}(\Omega)$ means $f_{t},\, D_{i}f,\, D_{ij}f\in\left(L_{x}^{p}L_{t}^{p}\right)_{loc}$. 
\end{defn}

\begin{thm}\label{variant}
	Under the assumptions above, there are constants $p:=p(\nu,\, n,\, S)<n+1$ and $N$ depending on \textup{$\nu$, $n$ and $S$} such that if $u\in C(\overline{\Omega})\cap W_{p}(\Omega)$ satisfies $-u_{t}+Lu\geq f$ in $\Omega$ with $f\in L^{p}(\Omega)$, and $u\leq0$ on $\partial_{p}\Omega$, then 
	\begin{equation}
	\sup_{\Omega}u\leq Nr^{2-\frac{n+2}{p}}||f||_{L^{p}}\label{abpkt}
	\end{equation}
where $r$ is the diameter of $\Omega$.
\end{thm}

With the help of the variant of Aleksandrov-Bakelman-Pucci-Krylov-Tso estimate \eqref{abpkt}, we can obtain:
\begin{thm}[Interior Harnack Inequality]\label{harnack}
	Suppose  $u\in C(\overline{Q_{2r}(Y)})\cap W_{p}(Q_{2r}(Y))$ and $-u_{t}+Lu=0$ in $Q_{2r}(Y)$,
	$Y=(y,s)\in\mathbb{R}^{n+1}$ and $r>0$. If $u\geq0$, then
	\begin{equation}
	\sup_{Q^{0}}u\leq N\inf_{Q_{r}}u,\label{eq:harc}
	\end{equation}
	where $N=N(n,\nu,S)$ and $Q^{0}=B_{r}(y)\times(s-3r^{2},s-2r^{2})$.
\end{thm}

Harnack inequalities have many important applications, not only in
differential equations, but also in other areas, such as diffusion
processes, geometry, etc. Unlike the classical maximum principle,
the interior Harnack inequality is far from obvious. For elliptic
and parabolic equations with measurable coefficients in the divergence
form, it was proved by Moser in the papers \cite{M61},\cite{M64}.
However, a similar result for non-divergence equations was obtained
15 years later after Moser's papers by Krylov and Safonov \cite{KS},
\cite{S80} in 1978-80. Their proofs relied on some improved versions
of growth theorems from the book by Landis \cite{EML}. These growth
theorems control the behavior of (sub-, super-) solutions of second
order elliptic and parabolic equations in terms of the Lebesgue measure
of areas in which solutions are positive or negative. So certainly, if some
estimate can directly import the information about measure, it should
be useful. In \cite{FS}, Ferretti and Safonov used growth theorems as a common
background for both divergence and non-divergence equations and used
these three growth theorems to derive the interior Harnack inequality.
Even in the one-dimensional case, the Harnack inequality fails for
equations of a ``joint'' structure,
which combine both divergence and non-divergence parts. One can find
detailed discussion in \cite{CS13}.

At the beginning, the interior Harnack inequality was proved with bounded drift. Later
on, this condition was relaxed to subcritical drift $b$. For the subcritical case, we can always rescale the
problem. In small scale, the drift will work like a perturbation from the case without drift. But for the critical situation, our common tricks do not work. One can
find a historical overview of this progress in \cite{NU}.
For non-divergence elliptic equations of second order, in \cite{S10},
Safonov shown the interior Harnack inequality for the scaling critical
case $b\in L^{n}$.  In \cite{GC},
the author proved the interior Harnack inequality for parabolic equations of second order in non-divergence form with the drift $b$   in certain Lebesgue spaces which are scaling invariant. In this note which appears as a companion of the earlier paper \cite{GC}, we consider the case when the drift $b\in M_{n+1,n+1}^{\frac{1}{n+1}}(\Omega)$ which is again scaling invariant. Similar results for both divergence
form elliptic and parabolic equations are presented in \cite{NU}.

We will follow the unified approach to growth theorems and the interior
Harnack inequality developed in \cite{FS}. For this purpose, we need to prove three growth
theorems and derive the interior Harnack inequality as a consequence
for parabolic equations with critical drift formulated as above. In order to take the measure conditions into account
and help us carry out growth theorems, we discuss the variant of Aleksandrov-Bakelman-Pucci-Krylov-Tso estimate \eqref{abpkt}. Although we only consider the case $b\in M_{n+1,n+1}^{\frac{1}{n+1}}(\Omega)$, one can see from the proofs, our approach works well for other pairs $(p,q,\alpha)$ satisfying condition \eqref{eq:04} with $\alpha>0$ provided the associated standard Aleksandrov-Bakelman-Pucci-Krylov-Tso estimate holds, see Sections 2 and 3. For the sake of simplicity, we assume that all functions
(coefficients and solutions) are smooth enough. It is easy to get
rid of extra smoothness assumptions by means of standard approximation
procedures, see Section 7. We should notice that it is important to have appropriate
estimates for solutions with constants depending only on the prescribed
quantities, such as the dimension $n$, the parabolicity constant, etc.,
but not depending on ``additional'' smoothness.

This paper is organized as follows: In Section 1, we introduce our
basic assumptions and notations. In Section 2, we formulate a weak
version of the classical maximum principle, the Alexandrov-Bakelman-Pucci-Krylov-Tso
estimate, and some consequences of it. In Section 3, we establish
a variant of Aleksandrov-Bakelman-Pucci-Krylov-Tso based on some estimates
of Green's function. In Sections 4, 5, 6, we formulate and prove three
growth theorems and prove the interior Harnack inequality. Finally,
in Section 7, we use approximation to show all results are valid without
smoothness assumption. In the appendix, an example of loss of continuity of the solution to a parabolic equation with drift $b\in L_{x,t}^{n+1}$ will be presented.

\subsection{Notations:}

In this paper, we use summation convention.

``$A:=B$'' or ``$B=:A$'' is the definition of $A$ by means of
the expression $B$.

$\mathbb{R}^{n}$ is the n-dimensional Euclidean space, $n\geq1$,
with points $x=(x_{1},\ldots,x_{n})^{t}$, where $x_{i}$'s are real
numbers. Here the symbol $t$ stands for the transposition of vectors
which indicates that vectors in $\mathbb{R}^{n}$ are treated as column
vectors. For $x=(x_{1},\ldots,x_{n})^{t}$ and $y=(y_{1},\ldots,y_{n})^{t}$
in $\mathbb{R}^{n}$, the scalar product $(x,y):=\Sigma x_{i}y_{i}$,
the length of $x$ is $|x|:=(x,x)^{\frac{1}{2}}$.

For a Borel set $\Gamma\subset\mathbb{R}^{n}$, $\bar{\Gamma}:=\Gamma\cup\partial\Gamma$
is the closure of $\Gamma$, $|\Gamma|$ is the n-dimensional Lebesgue
measure of $\Gamma$. Sometimes we use the same notation for the surface
measure of a subset $\Gamma$ of a smooth surface $S$.

For real numbers $c$, we denote $c_{+}:=\max(c,0)$, $c_{-}:=\max(-c,0)$.

In order to formulate our results, we need some standard definitions
and notations for the setting of parabolic equations.
\begin{defn}\label{cylinder}
	Let $Q$ be an open connected set in $\mathbb{R}^{n+1}$, $n\geq1$.
	The parabolic boundary $\partial_{p}Q$ of $Q$ is the set of all
	points $X_{0}=(x_{0},t_{0})\in\partial Q$, such that there exists
	a continuous function $x=x(t)$ on the interval $[t_{0},t_{0}+\delta)$
	with values in $\mathbb{R}^{n}$, such that $x(t_{0})=x_{0}$ and $(x(t),t)\in Q$
	for all $t\in(t_{0},t_{0}+\delta)$. Here $x=x(t)$ and $\delta>0$
	depend on $X_{0}$. In particular, for cylinders $Q_{U}=U\times(0,T)$ with $U\subset \mathbb{R}^n$,
	the parabolic boundary $\partial_{p}Q_{U}:=\left(\partial_{x}Q_{U}\right)\cup\left(\partial_{t}Q_{U}\right)$,
	where $\partial_{x}Q_{U}:=(\partial U)\times(0,T)$, $\partial_{t}Q_{U}:=U\times\{0\}$.
\end{defn}
We will use the following notation for the "standard" parabolic cylinder. For $Y=(y,s)$ and $r>0$, we define $Q_{r}(Y):=B_{r}(y)\times(s-r^{2},s),$
where $B_{r}(y):=\{x\in\mathbb{R}^{n}:|x-y|<r\}$.

\section{Preliminaries}

In this section, we briefly discuss some well-known theorems and results
which are crucial for us to carry out the discussion in the later
parts of this paper.  We use the notation $u\in W_{n+1}$ in the sense of Definition \ref{space}.
\begin{thm}[Aleksandrov-Bakelman-Pucci-Krylov-Tso
	estimate]
\label{thm:(Alexandrov-Bakelman-Pucci-Krylo}\label{thm:ABPKT} Suppose $u\in W_{n+1}(\Omega)$, $\Omega\subset Q_{r}$ and $-u_{t}+Lu\geq f$.  If
$\sup_{\partial_{p}\Omega}u\leq0$, then
\begin{equation}
\sup_{\Omega}u\leq N\left(r^{\frac{n}{n+1}}+\left\Vert b\right\Vert _{L^{n+1}}^{n}\right)||f||_{L^{n+1}}\label{eq:AK-1}
\end{equation}
where $N=N(n,\nu)$.
\end{thm}
One can find the detailed proof of the above standard version of Aleksandrov-Bakelman-Pucci-Krylov-Tso
estimate in \cite{AIN} and \cite{GL2}. From \eqref{eq:AK-1}, one
can see if we take $r^{\frac{n}{n+1}}$ out of the bracket, we will
have
\[
\sup_{\Omega}u\leq\sup_{\partial\Omega}u+Nr^{\frac{n}{n+1}}\left(1+\left\Vert b\right\Vert^{n} _{M_{n+1}^{\frac{1}{n+1}}(\Omega)}\right)||f||_{L^{n+1}}.
\]
So it is natural to consider the case $b\in M_{n+1}^{\frac{1}{n+1}}(\Omega)$.
\begin{rem}
In \cite{AIN}, Nazarov shown the Aleksandrov-Bakelman-Pucci-Krylov-Tso
estimate holds for the drift $b\in L_{x}^{p}L_{t}^{q}$, i.e.,
\[
\left\Vert b\right\Vert _{L_{x}^{p}L_{t}^{q}}=\left(\intop\left[\int\left|b(x,t)\right|^{q}dt\right]^{\frac{p}{q}}dx\right)^{\frac{1}{p}}<\infty,
\]
for
\[
\frac{n}{p}+\frac{2}{q}\leq1,\,\,\,\,p,q\geq1.
\]
The proof was based on Krylov\textquoteright{}s ideas and methods \cite{NVK}.
We believe if $b$ is in other scaling invariant Morrey space, and
the Aleksandrov-Bakelman-Pucci-Krylov-Tso estimate holds for the corresponding
case, then the proofs in this note also hold.\end{rem}

\begin{thm}[Maximal Principle]\label{thm:MP}
	Let $Q$ be a bounded open set
	in $\mathbb{R}^{n+1}$, and let a function $u\in C^{2,1}\left(\bar{Q}\backslash\partial_{p}Q\right)\cap C(\bar{Q})$
	satisfy the inequality $-u_{t}+Lu\geq0$ in $Q$. Then
	\begin{equation}
	\sup_{Q}u=\sup_{\partial_{p}Q}u\label{eq:-1}
	\end{equation}
	
\end{thm}
As an easy consequence of the maximal principle and the Alexandrov-Bakelman-Pucci-Krylov-Tso
estimate, we have the well-known comparison principle.
\begin{thm}[Comparison Principle]
	\label{thm:Comparison}Let $Q$ be a bounded
	domain in $\mathbb{R}^{n+1}$, $u,\, v\in C^{2,1}\left(\bar{Q}\backslash\partial_{p}Q\right)\cap C(\bar{Q})$,
	$-u_{t}+Lu\leq-v_{t}+Lv$ in $Q$, and $u\geq v$ on $\partial_{p}Q$,
	then $u\geq v$ on $\bar{Q}$.
\end{thm}

\section{A variant of Aleksandrov-Bakelman-Pucci-Krylov-Tso Estimate}

\subsection{Estimates of Green's function: }

In this subsection, we show some estimates for Green's function following
\cite{FSt,GL} in order to show the variant of Aleksandrov-Bakelman-Pucci-Krylov-Tso
estimate \eqref{abpkt}.

We consider the non-divergence second order parabolic equations of
the form

\begin{equation}
-u_{t}+Lu=-u_{t}+\sum_{i,j=1}^{n}a_{ij}D_{ij}u+\sum_{i=1}^{n}b_{i}D_{i}u\label{eq:21}
\end{equation}
defined on some Lipschitz domain $\Omega\subset\mathbb{R}^{n+1}$.

 We define Green's function $G$ : $\Omega\times\Omega\rightarrow\mathbb{R}$
satisfies following properties: if $u(x,t)$ has the form

\begin{equation}
u(x,t)=\intop_{\Omega}G(x,t,y,s)f(y,s)\,dyds\label{eq:22}
\end{equation}
 then it solves the Dirichlet problem:
\[
-u_{t}+Lu=-f,
\]
and
\[
u=0
\]
on $\partial_{p}\Omega$. Throughout this subsection, we will assume
all coefficients are smooth, then the existence of Green's function
is guaranteed. We will adapt the ideas in \cite{GL} to our parabolic
setting. The key step is to verify Lemma 2.1 in \cite{GL} holds in
parabolic case with $b\in M_{n+1}^{\frac{1}{n+1}}(\Omega)$. We formulate
the following lemma which is similar to Lemma 2.1 in \cite{GL} under
condition \eqref{eq:23}.
\begin{lem}
\label{lem:Green1}Let $(x_{1},t_{1})\in\Omega$, then we can choose
$1>\rho>0$ depending on $\nu$ ,$n$ and $S$ such that $Q_{2\rho}(x_{1},t_{1})\subset\Omega$ and
there is a constant $C$ depends on $\nu$, $n$ and $S$, we have

\begin{equation}
\left(\intop_{Q_{\rho}(x_{1},t_{1})}G(x,t,y,s)^{\frac{n+1}{n}}\,dyds\right)^{\frac{n}{n+1}}\leq\frac{C}{\rho^{\frac{n+2}{n+1}}}\intop_{Q_{\frac{\rho}{2}}(x_{1},t_{1})}G(x,t,y,s)\,dyds\label{eq:25}
\end{equation}
for any $(x,t)\in\Omega$ and $t\leq t_{1}$ where $G$ is Green's
function defined as \eqref{eq:22}.\end{lem}
\begin{proof}
Without loss of generality, we may assume $(x_{1},t_{1})=(0,0)$.
We will also use $Q_{r}$ to denote $Q_{r}(0,0)$. Clearly, it will
be sufficient to show that

\begin{equation}
\intop_{Q_{\rho}}G(x,t,y,s)f(y,s)\,dyds\leq\frac{C}{\rho^{\frac{n+2}{n+1}}}\intop_{Q_{\frac{\rho}{2}}}G(x,t,y,s)\,dyds\label{eq:26}
\end{equation}
for any non-negative function $f\in L^{n+1}(Q_{\rho})$ with $\intop_{Q_{\rho}}f^{n+1}=1$.
We fix such a $f$, and then construct $u_{1}$ and $u_{2}$ as following:

\begin{equation}
u_{1}(x,t)=\intop_{Q_{\rho}}G(x,t,y,s)f(y,s)\,dyds,\label{eq:27}
\end{equation}

\begin{equation}
u_{2}(x,t)=\intop_{Q_{\frac{\rho}{2}}}G(x,t,y,s)\,dyds.\label{eq:28}
\end{equation}
 We define
\begin{equation}
\eta(x,t):=1-\frac{1}{4\rho^{2}}\left\Vert x\right\Vert ^{2}-\frac{1}{4\rho^{2}}t.\label{eq:29}
\end{equation}
There are positive constants $N_{1}$, $\delta$ and $q$ (determined
by the parabolicity $\nu$ and the dimension $n$), such that

\begin{equation}
-D_{t}\eta^{q}+\sum_{ij}a_{ij}D_{ij}\eta^{q}\geq\begin{cases}
0 & Q_{2\rho}\backslash Q_{\frac{\rho}{2}}\\
-N_{1}\rho^{-2} & Q_{\frac{\rho}{2}}
\end{cases},\label{eq:210}
\end{equation}
and
\begin{equation}
\eta^{q}>\delta\label{eq:211}
\end{equation}
 in $Q_{\rho}$. \\
Let $u$ be the solution of the Dirichlet problem
\begin{equation}
-u_{t}+Lu=-f
\end{equation}
in $Q_{2\rho}$ and
\begin{equation}
u=0
\end{equation}
on $\partial_{p}Q_{2\rho}$. 

By the Aleksandrov-Bakelman-Pucci-Krylov-Tso
estimate we have
\begin{equation}
u\leq NN_{2}(\rho)^{\frac{n}{n+1}}=k_{1}\rho^{\frac{n}{n+1}},\label{eq:212}
\end{equation}
 where $k_{1}$ depends on $\nu$, $S$ and $n$. 
 
We set $C=\frac{2N_{1}NN_{2}}{\delta}$,
\begin{equation}
w:=u_{1}-\frac{C}{\rho^{\frac{n+2}{n+1}}}u_{2,}\label{eq:213}
\end{equation}

\begin{equation}
\bar{w}:=w-u+\frac{2NN_{2}\rho^{\frac{n}{n+1}}}{\delta}\eta^{q}\label{eq:214}
\end{equation}
 and

\begin{equation}
M:=\max\left\{ 0,\sup_{\partial_{p}Q_{2\rho}}w\right\} .\label{eq:215}
\end{equation}
 Then by some computations, we know that
\begin{equation}
-\overline{w}_{t}+L\overline{w}\geq N_{3}|b|\label{eq:216}
\end{equation}
 in $Q_{2\rho}$ where $N_{3}$ depends on $\nu$, $S$ and $n$.
And $\overline{w}\leq M$ on the parabolic boundary of $Q_{2\rho}$.
By the Aleksandrov-Bakelman-Pucci-Krylov-Tso estimate again, we have
\begin{equation}
\overline{w}\leq NN_{2}\rho S{}^{\frac{1}{n+1}}+M\label{eq:217}
\end{equation}
 in $Q_{2\rho}$. In $Q_{\rho}$ with $\eta^{q}>\delta$, we obtain

\begin{equation}
w\leq NN_{2}\rho S{}^{\frac{1}{n+1}}+M+NN_{2}(\rho)^{\frac{n}{n+1}}-\frac{2NN_{2}\rho^{\frac{n}{n+1}}}{\delta}\eta^{q}\leq M\label{eq:218}
\end{equation}
when we take $\rho$ small enough since $\rho<\rho{}^{\frac{n}{n+1}}$
when $\rho<1$. Here the smallness condition only depends on prescribed
constants.

By our construction, it is clear that
\begin{equation}
-w_{t}+Lw=0\label{eq:219}
\end{equation}
 in $\Omega\backslash Q_{\rho}$.
\begin{equation}
w\leq M\label{eq:220}
\end{equation}
 in $\overline{Q_{\rho}}$ and the parabolic boundary of $\Omega$.
By the maximal principle, we have $M=0$.  So $w\leq0$ in all $(x,t)\in\Omega$
with $t\leq0$. Therefore, after we decipher $w$, we get for $\rho$
small enough (here the smallness condition only depends on $\nu$,
$S$ and $n$), we conclude that
\[
\intop_{Q_{\rho}}G(x,t,y,s)f(y,s)\,dyds\leq\frac{C}{\rho^{\frac{n+2}{n+1}}}\intop_{Q_{\frac{\rho}{2}}}G(x,t,y,s)\,dyds,
\]
which implies
\[
\left(\intop_{Q_{\rho}(x_{1},t_{1})}G(x,t,y,s)^{\frac{n+1}{n}}\,dyds\right)^{\frac{n}{n+1}}\leq\frac{C}{\rho^{\frac{n+2}{n+1}}}\intop_{Q_{\frac{\rho}{2}}(x_{1},t_{1})}G(x,t,y,s)\,dyds.
\]

\end{proof}
With Lemma \ref{lem:Green1}, we can proceed to the prove of
the variant of Aleksandrov-Bakelman-Pucci-Krylov-Tso estimate \eqref{abpkt} similar
to results in \cite{GL} and \cite{FSt}. Since every quantity we
are considering here is scaling invariant, we may rescale our setting
to a domain with diameter $1$. So it suffices estimate
the integrability of Green's function in a domain with diameter $1$.
\begin{thm}
\label{thm:Green2}Under the same assumptions as above, then there are
constants $q:=q((\nu,\, n,\, S)>\frac{n+1}{n}$ and $C$ only depending
on $\nu$, $n$ and $S$ such that

\begin{equation}
\left(\intop_{\Omega}G(x,t,y,s)^{q}\,dyds\right)^{\frac{1}{q}}\leq C\label{eq:221}
\end{equation}
 where $G$ is Green's function and the diameter of $\Omega$ is $1$.\end{thm}
\begin{proof}
By our assumptions $\Omega$ is bounded. Suppose
\begin{equation}
\Omega'=\left\{ X:=(x,t)\in\Omega,\, d_{\Omega}(X)<\frac{1}{2}\right\} \label{eq:222}
\end{equation}
where

\[
d_{\Omega}(X):=\sup\left\{ \rho>0:\, Q_{\rho}(X)\subset\Omega\right\} .
\]
For arbitrary $(x',t')\in\Omega'$ and let $r<\frac{1}{2}$, then
$Q_{r}(x',t')\subset\Omega$. For each $Q_{r}(x',t')$, we can use
finite many $Q_{\rho}(x_{i},t_{i})$ to cover it, where $\rho$ is
chosen small enough so that the conditions of Lemma \ref{lem:Green1}
are satisfied. It is clear that the number of $Q_{\rho}(x_{i},t_{i})$
can be bounded by a positive constant $C_{1}$ depending on $\nu$, $n$
and $S$ and the diameter of $\Omega$.
Let $\rho=C_{2}r$, then we have

\begin{equation}
\left(\intop_{Q_{r}(x',t')}G(x,t,y,s)^{\frac{n+1}{n}}\,dyds\right)^{\frac{n}{n+1}}\leq\frac{C_{3}}{\rho^{\frac{n+2}{n+1}}}\intop_{Q_{r}(x',t')}G(x,t,y,s)\,dyds\label{eq:223}
\end{equation}
where $C_{3}$ depends on $C_{1}$, $C_{2}$ and $C$ in Lemma \ref{lem:Green1}.

We rewrite then inequality \eqref{eq:223} as
\begin{equation}
\left(\fint_{Q_{r}(x',t')}G(x,t,y,s)^{\frac{n+1}{n}}\,dyds\right)^{\frac{n}{n+1}}\leq\ C_{3}\fint_{Q_{r}(x',t')}G(x,t,y,s)\,dyds\label{eq:2233}
\end{equation}

By Gehring's lemma \cite{MM}, we get

\begin{equation}
\left(\intop_{\Omega'}G(x,t,y,s)^{q}\,dyds\right)^{\frac{1}{q}}\leq C_{4}\intop_{\Omega'}G(x,t,y,s)\,dyds\label{eq:224}
\end{equation}
for some
\[
q>\frac{n+1}{n},
\]
where $C_{4}$ depends on $C_{1}$, $C_{2}$ and $C$ in Lemma \ref{lem:Green1}.
\begin{equation}
\intop_{\Omega'}G(x,t,y,s)\,dyds< C_{5}
\end{equation}
where the constant $C_{5}$
depends on $\nu$, $n$ and $S$ by the standard Aleksandrov-Bakelman-Pucci-Krylov-Tso
estimate, Theorem \ref{thm:(Alexandrov-Bakelman-Pucci-Krylo}. Hence there is a constant $C'$ such that

\begin{equation}
\left(\intop_{\Omega'}G(x,t,y,s)^{q}\,dyds\right)^{\frac{1}{q}}\leq C'\label{eq:225}
\end{equation}
where $C'$ only depends on $\nu$, $n$ and $S$. In order to get
the same result for $\Omega$, we can extend the all coefficients
to a domain $\widetilde{\Omega}$ with comparable quantities, such
that $\Omega\subset\widetilde{\Omega}$ and $\forall X:=(x,t)\in\Omega$,
$d_{\widetilde{\Omega}}(X)<1$. Then Green's function $\widetilde{G}$
for $\widetilde{\Omega}$ will satisfy the same result as \eqref{eq:225}.
And we know $\widetilde{G}\geq G$ by the maximal principle and the
comparison principle. Hence we have

\begin{equation}
\left(\intop_{\Omega}G(x,t,y,s)^{q}\,dyds\right){}^{\frac{1}{q}}\leq C\label{eq:226}
\end{equation}
where $C$ depends on $\nu$, $n$ and $S$. \end{proof}
\begin{rem}
We can see from the proof, we just need $b\in M_{n+1}^{\alpha}(\Omega)$
for $\alpha$ positive, the above arguments work too provided the associated Aleksandrov-Bakelman-Pucci-Krylov-Tso Estimate holds. 
The point we
choose $\alpha=\frac{1}{n+1}$ is that the space is scaling invariant under
the parabolic scaling.
\end{rem}

\subsection{A variant of Aleksandrov-Bakelman-Pucci-Krylov-Tso Estimate:}

In this subsection, we still keep all the assumptions above. We will
obtain a variant of Aleksandrov-Bakelman-Pucci-Krylov-Tso estimate for
a solution of the second order parabolic equation differential equation
$-u_{t}+Lu=f$ in a bounded Lipschitz domain in $\mathbb{R}^{n+1}$
with Dirichlet problem condition. We estimate the maximal of $u$
in terms of the Dirichlet data (boundary value) on the boundary and
the $L^{p}$ norm of $f$ with $p<n+1$. For elliptic and parabolic
equations without drift, one can find some references in \cite{CS,EML,GL,FSt}.
\begin{thm*}
Under the same assumptions above, define $p=s \frac{q}{q-1}$, where $q>\frac{n+1}{n}$ is the constant from Theorem
\ref{thm:Green2}. Suppose  $u\in C(\overline{\Omega})\cap W_{p}(\Omega)$ satisfies $-u_{t}+Lu\geq f$ in $\Omega$ and $u\leq0$ on $\partial_{p}\Omega$. Then there is a constant $N$ depends
on \textup{$\nu$, $n$ and $S$} such that

\begin{equation}
\sup_{\Omega}u\leq Nr^{2-\frac{n+2}{p}}||f||_{L^{p}}\label{eq:31}
\end{equation}
 where $r$ is the diameter of $\Omega$. \end{thm*}
\begin{proof}
Without loss of generality, we may assume the diameter of $\Omega$
is 1. We first prove the result holds in case 1: when $u$ is smooth
and the coefficients of the equations are smooth (H\"older). Then
we use approximation to show in case 2: under general condition, the
estimate \eqref{eq:31} holds.\\
\emph{Case 1}\emph{\noun{:}} we can represent $u$ using:

\begin{equation}
u(x,t)=\intop_{\Omega}G(x,t,y,s)\left[u_{t}-Lu\right]\,dyds+\intop_{\partial_{p}\Omega}\tilde{G}(x,t,y,s)u\,dyds\label{eq:32}
\end{equation}
where $G$ is Green's function and $\tilde{G}$ is from the Riesz
representation theorem, since $\Omega$ is a Lipchitz domain. It is
clear both of $G$ and $\tilde{G}$ are non-negative, hence under
our assumptions we have

\begin{equation}
u(x,t)\leq\intop_{\Omega}G(x,t,y,s)\left|f\right|dyds\label{eq:33}
\end{equation}
then by Theorem \ref{thm:Green2} and H\"older inequality, it is clear
that
\begin{equation}
\sup_{\Omega}u\leq N\left\Vert f\right\Vert {}_{L^{p}},\label{eq:34}
\end{equation}
where $N$ depends on $\nu$, $n$ and $S$. \\
\emph{Case 2}\emph{\noun{:}} First we assume $u\in C(\overline{\Omega})\cap W_{p}^{2,1}(\Omega)$
and other functions are still smooth. Under these conditions, we pick
up a sequence $u^{i}$ which is smooth such that $u^{i}\rightarrow u$.
It is clear there is a sequence $\{f_{i}\}_{i}^{\infty}$ such that $-(u_{i})_{t}+Lu_{i}\geq f_{i}$
and $f_{i}\rightarrow f$ in $L^{p}(\Omega)$. So it is clear that
\begin{equation}
\sup_{\Omega}u\leq N||f||_{L^{p}}.\label{eq:35}
\end{equation}
Finally, in general situation. Suppose $a_{ij}^{k}\rightarrow a_{ij}$ and $b_{i}^{k}\rightarrow b_{i}$
almost everywhere as $k\rightarrow\infty$. Define
\begin{equation}
L_{k}u=\sum_{ij}a_{ij}^{k}D_{ij}u+\sum b_{i}^{k}D_{i}u,
\end{equation}
and
\begin{equation}
f_{k}=f+(L-L_{k})u.\label{eq:36}
\end{equation}
It is clear that $f_{k}\rightarrow f$ in $L^{p}(\Omega)$. Therefore
\begin{equation}
\sup_{\Omega}u\leq N||f||_{L^{p}}\label{eq:37}
\end{equation}
 holds.

Finally, after rescaling, we obtain \eqref{eq:31} in the most general
setting,
\[
\sup_{\Omega}u\leq Nr^{2-\frac{n+2}{p}}||f||_{L^{p}}.
\]

\end{proof}

\section{First Growth Theorem}

Suppose $R$ is the region in a cylinder where a subsolution $u$ of our equation is positive. The first growth theorem, Theorem \ref{thm:(First-Growth-Theorem)}, basically tells us if the measure of $R$ is small, then
the maximal value of $u$ over half of the cylinder is strictly less than the
maximal value over the whole cylinder. In other words, it gives us
some  quantitative decay properties. The variant of Aleksandrov-Bakelman-Pucci-Krylov-Tso
estimate \eqref{abpkt} enables us to import information about the measure into our
estimates.
\begin{thm}[First Growth Theorem]
\label{thm:(First-Growth-Theorem)} Let a function
$u\in C^{2,1}(\overline{Q_{r}})$ where $r>0$ and $Q_{r}=Q_{r}(Y)$,
in $\mathbb{R}^{n+1}$ containing $Y:=(y,s)$. Suppose $-u_{t}+Lu\geq0$
in $Q_{r}$, then $\forall\beta_{1}\in(0,1)$, there exists $0<\mu<1$
such that if we know
\begin{equation}
\left|\left\{ u>0\right\} \cap Q_{r}(Y)\right|\leq\mu|Q_{r}(Y)|,\label{eq:129}
\end{equation}
then

\begin{equation}
\mathcal{M}_{\frac{r}{2}}(Y)\leq\beta_{1}\mathcal{M}_{r}(Y),\label{eq:130}
\end{equation}
where
\[
\mathcal{M}_{r}(Y):=\max_{Q_{r}(Y)}u_{+}
\]
We also notice that $\beta_{1}\rightarrow0^{+}$ as $\mu\rightarrow0^{+}$.\end{thm}
\begin{rem}
First of all, we make some reductions. In our problem, we want
to show under some conditions, given $-u_{t}+Lu\geq0$ in a cylinder
$Q_{r}(Y)$, and some information about the set $\{u\leq0\}$, we
want to show that
\[
\mathcal{M}_{\frac{r}{2}}(Y)\leq\beta_{1}\mathcal{M}_{r}(Y).
\]
Clearly, in order to derive the above estimate, we only need to consider
positive part of $u$. We observe that to obtain the above estimate,
it actually suffices to show
\begin{equation}
u(Y)\leq\beta_{1}\mathcal{M}_{r}(Y),\label{eq:extra}
\end{equation}
for some $\beta_{1}\in(0,1)$. Indeed, for an arbitrary point $Z\in Q_{\frac{r}{2}}(Y)$,
we notice $Q_{\frac{r}{2}}(Z)\subset Q_{r}(Y)$, we can apply the
above estimate \eqref{eq:extra} to $Q_{\frac{r}{2}}(Z)$ with $Y$
replaced by $Z$ and $r$ replaced by $\frac{r}{2}$ with some measure
condition $\mu'$. In consistent with the measure condition in the
first growth theorem, we also observe that
\[
\left|\left\{ u>0\right\} \cap Q_{\frac{r}{2}}(Z)\right|\leq\left|\left\{ u>0\right\} \cap Q_{r}(Y)\right|\leq\mu\left|Q_{r}\right|=2^{n+2}\mu\left|Q_{\frac{r}{2}}(Z)\right|.
\]
So we just need to take $\mu=2^{-n-2}\mu'$ for the measure condition
in the first growth theorem.\end{rem}
\begin{proof}
Since every quantity is scaling invariant, we might assume $r=1$. And
we can multiply $u$ by a constant, so without loss of generality, we
can also assume $\mathcal{M}_{1}(Y)=1$. Also we assume $u(Y)>0$,
otherwise the result is trivial.
\begin{equation}
v(X)=v(x,t)=u(x,t)+t-s-|x-y|^{2}\label{eq:43}
\end{equation}
in $Q:=\left\{ v>0\right\} \cap Q_{1}(Y)$. Clearly, $Q\neq\emptyset$
since $v(Y)=u(Y)>0$ and $Y\in\partial Q_{1}(Y)$. It is easy to see
that $v\leq u$ in $Q$. By the measure condition, we have
\[
\left|Q\right|\leq\left|\left\{ u>0\right\} \cap Q_{1}(Y)\right|\leq\mu_{1}|Q_{1}(Y)|\leq\mu_{1}.
\]
Note that $v\leq0$ on $\partial_{p}Q_{1}(Y)$, so $v=0$ on $\partial_{p}Q$.
Since $-u_{t}+Lu\geq0$, we know that
\begin{equation}
(-\partial_{t}+L)v\geq0-1-2trace(a_{ij})-2|b|\geq-1-2nv^{-1}-2|b|.\label{eq:44}
\end{equation}
By the variant of Aleksandrov-Bakelman-Pucci-Krylov-Tso estimate \eqref{abpkt} with some constant $p<n+1$ and
H\"older inequality,
\begin{equation}
u(Y)\leq N(\nu,n,S)\left\Vert -1-2nv^{-1}-2|b|\right\Vert _{L^{p}(Q)}\leq N_{1}(\nu,n,S)\left(\mu^{\frac{1}{p}}+S^{\frac{1}{n+1}}\mu^{\frac{1}{p}-\frac{1}{n+1}}\right).\label{eq:45}
\end{equation}
Now we can pick $\mu$ small enough so that for fixed $\beta_{1}$,
then we have
\begin{equation}
u(Y)<\beta_{1}.\label{eq:46}
\end{equation}
It is also clear from the construction, $\beta_{1}\rightarrow0^{+}$
as $\mu\rightarrow0^{+}$.
\end{proof}
With the first growth theorem, we can do the following useful argument
which is helpful for us to find a non-degenerate point to build a
bridge between two regions we are interested in. Without loss of generality,
we still assume $r=1$, for $X\in Q_{1}(Y)$, we define
\begin{equation}
d(X):=\sup\left\{ \rho>0:\, Q_{\rho}(X)\subset Q_{1}(Y)\right\} .\label{eq:47}
\end{equation}
Roughly here $d$ plays roles of weights with which we can make
sure the point we are interested in is not degenerate, i.e., it is
in the interior of the cylinder. For $\gamma>0$,
we consider $d^{\gamma}u(x)$ instead of $u(x)$. $d^{\gamma}u(x)$
is a continuous function in $\overline{Q_{1}(Y)}$. Clearly, $d(Y)=1$, we obtain
\begin{equation}
u(Y)=d^{\gamma}u(Y)\leq M:=\sup_{Q_{1}(Y)}d^{\gamma}u.\label{eq:48}
\end{equation}
By our construction, $d^{\gamma}u$ vanishes on $\partial_{p}Q_{1}$,
so $\exists X_{0}\in\overline{Q_{1}(Y)}\backslash\partial_{p}Q_{1}$
such that
\begin{equation}
M=d^{\gamma}u(X_{0}).\label{eq:49}
\end{equation}
Let $r_{0}:=\frac{1}{2}d(X_{0})$, we consider the intermediate region
$Q_{r_{0}}(X_{0})$, In this region, we have
\[
\forall X\in Q_{r_{0}}(X_{0}),\,\, d(X)\geq r_{0}.
\]
Therefore, we conclude that
\begin{equation}
\sup_{Q_{r_{0}}(X_{0})}u\leq r_{0}^{-\gamma}\sup_{Q_{r_{0}}(X_{0})}d^{\gamma}u\leq r_{0}^{-\gamma}M\leq2^{\gamma}u(X_{0}).\label{eq:410}
\end{equation}
Now, we define $v=u-\frac{1}{2}u(X_{0})$, then
\[
v(X_{0})=\frac{1}{2}u(X_{0})\geq2^{-1-\gamma}\sup_{Q_{r_{0}}(X_{0})}u>2^{-1-\gamma}\sup_{Q_{r_{0}}(X_{0})}v.
\]
From the first growth theorem, Theorem \ref{thm:(First-Growth-Theorem)},
we know $\exists\mu(n,\nu,\gamma,S)\in(0,1]$ such that Theorem \ref{thm:(First-Growth-Theorem)} holds with $\beta_{1}=2^{-1-\gamma}$. Now the above
inequality tells us that $v$ does not satisfy the measure condition
in the first growth theorem. So
\begin{equation}
\left|\left\{ v>0\right\} \cap Q_{r_{0}}(X_{0})\right|=\left|\left\{ u>\frac{1}{2}u(X_{0})\right\} \cap Q_{r_{0}}(X_{0})\right|>\mu\left|Q_{r_{0}}(X_{0})\right|.\label{eq:411}
\end{equation}
Now, we can show an integral estimate which is equivalent to the first
growth theorem.
\begin{thm}
\label{thm:Lepsilon}Let a function $u\in C^{2,1}(Q_{r})$, where
$Q_{r}:=Q_{r}(Y)$, $Y=(y,s)\in\mathbb{R}^{n+1}$, $r>0$. If
$-u_{t}+Lu\geq0$ in $Q_{r}$, then for arbitrary $p>0$, we obtain
\begin{equation}
u_{+}^{p}(Y)\leq\frac{N}{\left|Q_{r}\right|}\intop_{Q_{r}}u_{+}^{p}\,dX,\label{eq:412}
\end{equation}
where $N$ only depends on $n,\,\nu,\, S,\, p$.\end{thm}
\begin{proof}
Since the quantities we are considering are scaling invariant, we
might assume $r=1$. By the similar argument as above, we choose
\begin{equation}
\gamma=\frac{n+2}{p}.\label{eq:413}
\end{equation}
We get
\begin{equation}
\left|\left\{ u>\frac{1}{2}u(X_{0})\right\} \cap Q_{r_{0}}(X_{0})\right|>\mu\left|Q_{r_{0}}(X_{0})\right|,\label{eq:414}
\end{equation}
where $\mu=\mu(n,\nu,\gamma,S)\in(0,1]$. With the same notations
as above we have
\begin{eqnarray}
u_{+}^{p}(Y) & \leq & M^{p}=\left(u\right)^{p}(X_{0})=\left(2r_{0}\right)^{\gamma p}u^{p}(X_{0})\nonumber \\
 & \leq & \frac{\left(2r_{0}\right)^{\gamma p}}{\left|\left\{ u>\frac{1}{2}u(X_{0})\right\} \cap Q_{r_{0}}(X_{0})\right|}\intop_{\left\{ u>\frac{1}{2}u(X_{0})\right\} \cap Q_{r_{0}}(X_{0})}(2u)^{p}\,dX\nonumber \\
 & \leq & \frac{2^{\gamma p+p}r_{0}^{\gamma p}}{\mu\left|Q_{r_{0}}(X_{0})\right|}\intop_{Q_{1}}u_{+}^{p}\,dX.\label{eq:415}
\end{eqnarray}
By our construction, it gives $r_{0}^{\gamma p}=r_{0}^{n+2}$, so the
above estimate implies
\begin{equation}
u_{+}^{p}(Y)\leq\frac{N}{\left|Q_{1}\right|}\intop_{Q_{1}}u_{+}^{p}dX.\label{eq:416}
\end{equation}

\end{proof}
We have seen the first growth theorem implies Theorem \ref{thm:Lepsilon}.
Actually, we can also obtain the first growth theorem from Theorem \ref{thm:Lepsilon}.
Indeed, from Theorem \ref{thm:Lepsilon}, we have
\begin{equation}
u_{+}^{p}(Y)\leq\frac{N}{\left|Q_{r}\right|}\intop_{Q_{r}}u_{+}^{p}\,dX\leq N\frac{\left|\left\{ u>0\right\} \cap Q_{r}(Y)\right|}{\left|Q_{r}\right|}\sup_{Q_{r}}u_{+}^{p}\leq\mu N\sup_{Q_{r}}u_{+}^{p}.\label{eq:417}
\end{equation}
Now it is trivial to see $\beta_{1}\rightarrow0^{+}$ as $\mu\rightarrow0^{+}$
which is in consistent with the conditions in the first growth theorem.
\begin{rem*}
The idea we used to find a non-degenerate point above will be also
helpful when we prove the interior Harnack inequality.
\end{rem*}

\section{Second Growth Theorem}

Before we establish the second growth theorem, Theorem \ref{thm:(Second-Growth-Theorem).}, we need to prove some intermediate results based on the comparison principle and the Aleksandrov-Bakelman-Pucci-Krylov-Tso
estimate. Let us first do some preliminary calculations in order to
carry out some comparison arguments.

For fixed number $\alpha>0$ and $0<\epsilon<1$, in the cylinder
$Q=B_{r}(0)\times(-r^{2},(\alpha-1)r^{2})$, we can define
\begin{equation}
\psi_{0}=\frac{(1-\epsilon^{2})(t+r^{2})}{\alpha}+\epsilon^{2}r^{2}\label{eq:51}
\end{equation}
and
\begin{equation}
\psi_{1}=\left(\psi_{0}-|x|^{2}\right)_{+}\label{eq:52}
\end{equation}
where $\left(\cdot\right)_{+}$ means positive part of the function.
And we also define
\begin{equation}
\psi=\psi_{1}^{2}\psi_{0}^{-q}\label{eq:53}
\end{equation}
for some number $q\geq2$ to be determined later. First of all, we
notice $\psi$ is $C^{2,1}$ in $\widetilde{Q}:=\left\{ (x,t)|\,|x|^{2}<\psi_{0},\,-r^{2}<t<(\alpha-1)r^{2}\right\} $.
It is clear that $-\psi_{t}+L\psi=0$ if $\psi_{0}\leq|x|^{2}$. Now
if $\psi_{0}>|x|^{2}$, by some computations, we obtain
\[
-\psi_{t}+a_{ij}D_{ij}\psi=\psi_{0}^{-q}\left[8a_{ij}x_{i}x_{j}-4\psi_{1}trace(a_{ij})+\frac{(1-\epsilon^{2})q}{\alpha\psi_{0}}\psi_{1}^{2}-2\frac{(1-\epsilon^{2})}{\alpha}\psi_{1}\right].
\]
Set $F_{1}=\frac{2}{\alpha}+8n\nu^{-1}$, and $\xi=\frac{\psi_{1}}{\psi_{0}}$
then

\begin{equation}
-\psi_{t}+a_{ij}D_{ij}\psi\geq\psi_{0}^{1-q}\left[\frac{(1-\epsilon^{2})q}{\alpha}\xi^{2}-F_{1}\xi+8\lambda\right].\label{eq:54}
\end{equation}
Pick
\begin{equation}
q=2+\frac{\alpha}{32(1-\epsilon^{2})},\label{eq:55}
\end{equation}
so that the quadratic form in \eqref{eq:54} is non-negative. Then we
can conclude that 
\[
-\psi_{t}+\sum_{ij}a_{ij}D_{ij}\psi\geq0,\,\,\,\,\forall(x,t)\in Q
\]
We also notice that
\[
\psi(x,-r^{2})\leq(\epsilon r)^{-2q+4},\,\,\,\,\forall|x|\leq r,
\]
and
\begin{equation}
\psi\left(x,(\alpha-1)r^{2}\right)\geq\frac{9}{16}r^{-2q+4},\,\,\,\,\,\forall|x|\leq\frac{r}{2}.\label{eq:56}
\end{equation}
Finally, we notice that by the monotonicity of $\psi$ with respect
to $t\in\left[-r^{2},(\alpha-1)r^{2}\right]$ for $x=0$, we obtain
\begin{equation}
\psi\left(0,t\right)\geq\frac{9}{16}r^{-2q+4}.\label{eq:57}
\end{equation}
Now consider $-u_{t}+Lu\leq0$ , $u>0$ in $\Omega$. In Lemma \ref{lem:samegrowth}, fist of all, we establish that at least
for a cylinder short enough, if we have a lower bound on some interior
portion of the bottom, then it has a quantitative lower bound for
the same portion on the top of the short cylinder. We also notice
that the shortness only depends on the prescribed constants
but not $u$. Then we can iterate this process to get a quantitative
lower bound for an arbitrary time.
\begin{lem}
\label{lem:samegrowth}Let $\alpha$ be a positive constant and $-u_{t}+Lu\leq0$,
$u>0$ in $\Omega$. Suppose $Q:=B_{r}(0)\times(-r^{2},(\alpha-1)r^{2})\subset\Omega$
and $u>0$ in $B_{r}(0)\times(-r^{2},(\alpha-1)r^{2})$. Then for $\epsilon<\frac{1}{2}$, there are positive constants $C_{1}=C_{1}(n,\nu)$ and $m=m(n,\nu,\alpha)$
such that if
\begin{equation}
u\geq\ell\label{eq:58}
\end{equation}
on $B_{\epsilon r}(0)\times\{-r^{2}\}$, then
\begin{equation}
u\geq C_{1}\epsilon^{m}\ell\label{eq:59}
\end{equation}
on $B_{\epsilon r}(0)\times\{(\alpha-1)r^{2}\}$.\end{lem}
\begin{proof}
\emph{\noun{Step 1:}}\noun{ }
\begin{equation}
-\psi_{t}+L\psi\geq b_{i}D_{i}\psi=-4\psi_{1}\psi_{0}^{-q}(b,x)\geq-4|b|r(\frac{1}{2})^{-2q}r^{2-2q}.\label{eq:510}
\end{equation}
Consider
\begin{equation}
v=u-\ell(\epsilon r)^{2q-4}\psi\label{eq:511}
\end{equation}
It is clear $\psi=0$ for $|x|=r$. So we can conclude that $v\geq0$
on $\partial_{p}\widetilde{Q}$ by the above calculations. Finally,
we apply the Aleksandrov-Bakelman-Pucci-Krylov-Tso estimate to $-v$,
we get
\begin{equation}
v\geq-N(n,\nu,S)\ell\epsilon{}^{-4}S\label{eq:512}
\end{equation}
 in $\widetilde{Q}$. In other words, we get
\begin{equation}
u\geq\ell(\epsilon r)^{2q-4}\psi-N(n,\nu,S)\ell\epsilon^{-4}S.\label{eq:513}
\end{equation}
So we have
\begin{equation}
u(x,(\alpha-1)r^{2})\geq\ell\epsilon{}^{-4}\left[\frac{9}{16}(\epsilon)^{2q}-N(n,\nu,S)S\right].\label{eq:514}
\end{equation}
In particular, with the monotonicity of $\psi$ with respect to $t$
when $x=0$, we obtain
\begin{equation}
u(0,t)\geq\ell\epsilon{}^{-4}\left[\frac{9}{16}(\epsilon)^{2q}-N(n,\nu,S)S\right],\,\forall t\in\left[-r^{2},(\alpha-1)r^{2}\right].\label{eq:515}
\end{equation}
By the similar calculations as above but with the variant of Aleksandrov-Bakelman-Pucci-Krylov-Tso
applied to the region $B_{r}(0)\times(-r^{2},(h\alpha-1)r^{2})$,
we get
\begin{equation}
u(0,(h\alpha-1)r^{2})\geq\ell\epsilon{}^{-4}\left[\frac{9}{16}(\epsilon)^{2q}-N(n,\nu,S)S^{\frac{1}{n+1}}h^{\frac{1}{p}-\frac{1}{n+1}}\right].\label{eq:516}
\end{equation}
Pick $h=h(n,\nu,S,\epsilon)$ small, we know that
\begin{equation}
\left[\frac{9}{16}(\epsilon)^{2q}-N(n,\nu,S)S^{\frac{1}{n+1}}(\alpha h)^{\frac{1}{p}-\frac{1}{n+1}}\right]\geq\frac{1}{2}\epsilon^{2q}.\label{eq:517}
\end{equation}
So we can conclude
\begin{equation}
u(0,t)\geq C_{1}\epsilon^{k}\ell\label{eq:518}
\end{equation}
for $t\leq(h-1)r^{2}$ where $k=2q-4$ and $C_{1}$ does not depend
on $u$.\\
\emph{\noun{Step 2:}}\noun{ }For arbitrary $t=(\alpha h-1)r^{2}$,
and $x\in B_{\epsilon}(0)$, we can use a slanted cylinder with radius
$\epsilon$ to connect $B_{\epsilon}(x)\times\{(\alpha h-1)r^{2}\}$
and $B_{\epsilon}(0)\times\{-r^{2}\}$. We can use a change of
coordinate to reduce the slanted cylinder to a regular cylinder. We
notice that with $k_{i}:=\frac{y_{i}}{s}$ and $\frac{|y_{i}|}{s}=|k_{i}|\leq\frac{|y|}{s}\leq K$,
then define $w_{i}=x_{i}-k_{i}t$ and $z=t$. In this coordinate,
the slanted cylinder is transformed to a standard cylinder. The equation
with respect to the new coordinate is
\begin{equation}
-u_{z}+\sum_{ij}a_{ij}D_{w_{i}w_{j}}u+\sum_{i}(b_{i}+k_{i})D_{w_{i}}u\leq0.\label{eq:519}
\end{equation}
Then we apply the standard cylinder results to the equation with respect
to coordinate $(w,z)$. We can do the same argument for all $x\in B_{\epsilon}(0)$.
We have find $h_{x}$ such that
\begin{equation}
u(x,(\alpha h_{x}-1)r^{2})\geq C_{2}\epsilon^{k}\ell.\label{eq:520}
\end{equation}
Since $K$ is uniformly bounded above, so indeed, $h_{x}$ and $C_{2}$
only depend on $\epsilon$, $n,\nu,S$. In particular, $h_{x}$ can
be uniformly bounded from below. Finally, we take
\begin{equation}
h_{0}=\inf_{x\in B_{\epsilon}(0)}h_{x}>0,\label{eq:521}
\end{equation}
then  for $x\in B_{\epsilon}(0)$, $t=(\alpha h_{0}-1)r^{2}$,
we obtain
\begin{equation}
u(x,t)\geq C_{2}\epsilon^{k}\ell.\label{eq:522}
\end{equation}

\end{proof}
\emph{\noun{Step 3:}}\noun{ }Now for the general case, let $\alpha$
be positive constant as above, we can pick $h_{0}$ based on our discussion
above. Finally by a simple iteration argument, we get the above result.
Therefore $\exists m=m(n,S,\nu,\alpha)$, such that at least we conclude
\begin{equation}
u\geq C_{1}\epsilon^{m}\ell\label{eq:523}
\end{equation}
on $B_{\epsilon}(0)\times\{(\alpha-1)r^{2}\}$.
\begin{rem}
In fact, in consistent with the Lemma 7.39 in \cite{GL2}, we can
show that
\[
u(x,(\alpha-1)r^{2})\geq C_{3}\epsilon^{m}\ell
\]
on $B_{\frac{1}{2}r}(0)\times\{(\alpha-1)r^{2}\}$.
\end{rem}

For a fixed point $Y=(y,s)\in\mathbb{R}^{n+1}$
with $s>0$, and $r>0$, we define the slanted cylinder
\begin{equation}
V_{r}=V_{r}(Y):=\left\{ X=(x,t)\in\mathbb{R}^{n+1};\,\left|x-\frac{t}{s}y\right|<r,0<t<s\right\} .\label{eq:122}
\end{equation}

Now the useful slanted cylinder lemma \cite{FS} follows easily from
Lemma \ref{lem:samegrowth} after we apply Lemma
\ref{lem:samegrowth} to $1-u$ after we multiply $u$ by a constant
to reduce our problem to the case $1=\sup_{V_{r}(Y)}u$.

\begin{lem}[Slanted Cylinder Lemma]\label{lem:(Slant-Cylinder-Lemma)} Let a function
$u\in C^{2,1}(\overline{V_{r}})$ satisfy $-u_{t}+Lu\geq0$ in a slanted
cylinder $V_{r}$, which is defined in \eqref{eq:122} with $Y=(y,s)\in\mathbb{R}^{n+1}$,
$s>0$, $r>0$, such that

\begin{equation}
K^{-1}r|y|\leq s\leq Kr^{2}\label{eq:525}
\end{equation}
where $K>1$ is a constant. In addition, suppose $u\leq0$ on $D_{r}:=B_{r}(0)\times\{0\}$.
Then

\begin{equation}
u(Y)\leq\beta_{2}\sup_{V_{r}(Y)}u_{+}\label{eq:526}
\end{equation}
with a constant $\beta_{2}=\beta_{2}(\nu,n,K,S)<1$.
\end{lem}
With the slanted cylinder lemma, we are ready to prove the second
growth theorem. The slanted cylinder lemma, i.e., Lemma \ref{lem:(Slant-Cylinder-Lemma)}
above plays a crucial role in this section to build a connection between
different time slides. The second growth theorem helps us construct
some control of the oscillation between different time slides. We
follow the arguments in \cite{FS}.
\begin{thm}[Second Growth Theorem]
\textup{\label{thm:(Second-Growth-Theorem).}
Let a function $u\in C^{2,1}\left(\overline{Q_{r}}\right)$, where
$Q_{r}:=Q_{r}(Y)$, $Y=(y,s)\in\mathbb{R}^{n+1}$, $r>0$, and let
$-u_{t}+Lu\geq0$ in $Q_{r}$. In addition, suppose $u\leq0$ on $D_{\rho}:=B_{\rho}(z)\times\{\tau\}$,
where $B_{\rho}(z)\subset B_{r}(y)$ and
\begin{equation}
s-r^{2}\leq\tau\leq s-\frac{1}{4}r^{2}-\rho^{2}.\label{eq:527}
\end{equation}
Then
\begin{equation}
u(Y)\leq\beta_{3}\sup_{Q_{r}(Y)}u_{+}\label{eq:528}
\end{equation}
where $\beta_{3}:=\beta_{3}(n,\nu,\rho/r,S)<1$ is a constant.}\end{thm}
\begin{proof}
After rescaling and translation in $\mathbb{R}^{n+1}$, we reduce
our problem to $r=1$, and $(z,\tau)=(0,0)\in\mathbb{R}^{n+1}$. For
an arbitrary point $Y'\in Q_{\frac{1}{2}}(Y)$, we can apply the slanted
cylinder lemma to the slanted cylinder $V_{\rho}(Y')\subset Q_{1}(Y)$.
Note that in this situation, the constant $K$ in slanted cylinder lemma
only depends on $\rho$. Therefore, with the parameter $\beta_{2}$
from the slanted cylinder lemma, we have
\[
u(Y')\leq\beta_{2}\sup_{V_{\rho}(Y')}u_{+}\leq\beta_{2}\sup_{Q_{1}(Y)}u_{+}.
\]
The above estimate holds for all $Y'\in Q_{\frac{1}{2}}(Y)$. Therefore
we get
\[
\sup_{Q_{\frac{1}{2}}(Y)}u_{+}\leq\beta_{2}\sup_{Q_{1}(Y)}u_{+}.
\]

\end{proof}
Now we establish an estimate similar to above with  more explicit
dependence of the constant on the ratio $\rho/r$.
\begin{lem}
\label{lem:quotient}Let a function $v\in C^{2,1}\left(\overline{Q_{r}}\right)$
satisfy $v\geq0$, $-v_{t}+Lv\leq0$ in $Q_{r}:=Q_{r}(Y),Y=(y,s)\in\mathbb{R}^{n+1},r>0$.
For arbitrary disks $D_{\rho}:=B_{\rho}(z)\times\{\tau\}$ and
$D^{0}:=B_{\frac{r}{2}}(y)\times\{\sigma\}$, such that $B_{\rho}(z)\subset B_{r}(y)$
and
\begin{equation}
s-r^{2}\leq\tau<\tau+h^{2}r^{2}\leq\sigma\leq s,\label{eq:529}
\end{equation}
where $h\in(0,1)$ is a constant. Then 
\begin{equation}
\inf_{D_{\rho}}v\leq\left(\frac{2r}{\rho}\right)^{\gamma}\inf_{D^{0}}v\label{eq:530}
\end{equation}
where $\gamma=\gamma(n,\nu,h,S)$.\end{lem}
\begin{proof}
Without loss of generality, we may assume $m:=\inf_{D_{\rho}}v>0$, $r=1$,
$z=0$, $\tau=0$, $\sigma=s=h^{2}$. So $D_{\rho}=B_{\rho}(0)\times\{0\}$.
We can apply an additional linear transformation along $t$-axis,
we can also reduce the proof to the case $h=1$. Now fix the integer
$k$ such that $2^{-k-1}<\rho\leq2^{-k}$, and for $j=0,1,\ldots$,
and define $y^{j}:=y^{*}+2^{-j}(y-y^{*})$, $B^{j}:=B_{2^{-j}}(y^{j})$,
where $y^{*}:=\frac{\rho}{1-\rho}y$, $Y^{j}:=\left(y^{j},4^{-j}\right)$,
$Q^{j}:=Q_{2^{-j}}(Y^{j})$, $D^{j}:=B_{2^{-j-1}}(y^{j})\times\{4^{-j}\}$.
By construction, $0=y^{*}+\rho(y-y^{*})$, so that
\[
B_{\rho}(0),\, B^{j}\in\{B_{\theta}\left(y^{*}+\theta(y-y^{*})\right);\,0\leq\theta\leq1\}.
\]
Then by the assumption, $B_{\rho}(0)\subset B_{1}(y)$ it follows
$\left|y\right|\leq1-\rho$, $\left|y-y^{*}\right|\leq1$, and
\[
B^{k+1}\subset B_{\rho}(0)\subset B^{k}\subset B^{k-1}\subset\ldots\subset B^{1}\subset B^{0}=B_{1}(y).
\]
Now apply Theorem \ref{thm:(Second-Growth-Theorem).} to the function
$u=1-\frac{1}{m}v$ in $Q^{k}$ with
\[
r=2^{-k},\,\rho=2^{-k-1},\, Y=Y^{k},\, z=0,\,\tau=0.
\]
Then we conclude that
\[
\sup_{D_{k}}u\leq\sup_{Q_{2^{-k-1}(Y^{k})}}u\leq\beta_{3}\sup_{Q^{k}}u\leq\beta_{3}=\beta_{2}(n,\nu,S,\frac{1}{2})<1,
\]
which is equivalent to
\[
\inf_{D_{\rho}}v=m\leq(1-\beta_{3})^{-1}\inf_{D^{k}}v=2^{\gamma}\inf_{D^{k}}v,
\]
where $\gamma:=-\log_{2}(1-\beta_{3})>0$. Similarly, if $k\geq1$,
we also have
\[
\inf_{D^{j}}v\leq2^{\gamma}\inf_{D^{j-1}}v,
\]
for $j=1,2,\ldots,k$. Finally we have
\[
\inf_{D_{\rho}}v\leq2^{\gamma}\inf_{D^{k}}v\leq2^{2\gamma}\inf_{D^{k-1}}v\leq\ldots\leq2^{(k+1)\gamma}\inf_{D^{0}}v\leq\left(\frac{2r}{\rho}\right)^{\gamma}\inf_{D^{0}}v.
\]

\end{proof}

\section{Interior Harnack Inequality}

We also need the third growth theorem in order to establish the interior
Harnack inequality. The first growth theorem tells us if $\mu\rightarrow0^{+}$ then $\beta_{1}\rightarrow0^{+}$.
The third growth theorem tells us if we have a nice control of the
measure of the set $\left\{ u>0\right\} $ near the bottom, then we
can have a more precise estimate. In other words, if we have the similar measure
condition for
\[
Q^{0}:=Q_{\frac{r}{2}}(Y^{0}), Y^{0}=\left(y,s-\frac{3}{4}r^{2}\right).
\]
Then if $\mu<1$, then $\beta_{1}<1$. The proof of it is long and technical but independent
of the specific structure of the equations. One can find a detailed proof
in, for example, \cite{FS,KS,GC}. We just formulate the results
here.
\begin{thm}[Third Growth Theorem]
\label{thm:(Third-Growth-Theorem).} Let a
function $u\in C^{2,1}\left(\overline{Q_{r}}\right)$, where $Q_{r}=Q_{r}(Y)$,
$Y=(y,s)\in\mathbb{R}^{n+1}$, $r>0$, and let $-u_{t}+Lu\geq0$ in
$Q_{r}$. In addition, we assume
\begin{equation}
\left|\left\{ u>0\right\} \cap Q^{0}\right|\leq\mu\left|Q^{0}\right|,\label{eq:71}
\end{equation}
where
\begin{equation}
Q^{0}:=Q_{\frac{r}{2}}(Y^{0}),\,\, Y^{0}=\left(y,s-\frac{3}{4}r^{2}\right)\label{eq:72}
\end{equation}
and $\mu<1$ is a constant. Then we have
\begin{equation}
\mathcal{M}_{\frac{r}{2}}(Y)\leq\beta\mathcal{M}_{r}(Y)\label{eq:73}
\end{equation}
with a constant
\[
\beta:=\beta(n,\nu,S,\mu)<1.
\]
\end{thm}
\begin{cor}
\label{lower}Let a function $v\in C^{2,1}\left(\overline{Q_{r}}\right)$
be such that $v\geq0$, $-v_{t}+Lv\leq0$ in $Q_{r}$, and
\begin{equation}
\left|\left\{ v\geq1\right\} \cap Q^{0}\right|>(1-\mu)\left|Q^{0}\right|.\label{eq:74}
\end{equation}
Then
\begin{equation}
v\geq1-\beta>0\label{eq:75}
\end{equation}
on $Q_{\frac{r}{2}}$ where $\beta=\beta(n,\nu,\mu,S)<1$ for $\mu<1$.
\end{cor}
With the above preparation, we now are ready to establish the interior
Harnack inequality.
\begin{thm*}[Interior Harnack Inequality]
 Suppose $u\in C^{2,1}(Q_{2r}(Y)) \cap C(\overline{Q_{2r}(Y)})$ and $-u_{t}+Lu=0$ in $Q_{2r}(Y)$,
$Y=(y,s)\in\mathbb{R}^{n+1}$ and $r>0$. If $u\geq0$, then
\begin{equation}
\sup_{Q^{0}}u\leq N\inf_{Q_{r}}u,\label{eq:76}
\end{equation}
where $N=N(n,\nu,S)$ and $Q^{0}=B_{r}(y)\times(s-3r^{2},s-2r^{2})$.
\end{thm*}
We will build a non-degenerate intermediate region to get a quantitative
relation between two regions we are interested in with the help of
three growth theorems.
\begin{proof}
After rescaling and translating as necessary,
we can assume $Y=0$ and $r=1$. Now $Q_{1}=B_{1}(0)\times(-1,0)$,
$Q^{0}=B_{1}(0)\times(-3,-2)$. It is easy to see that if we define
$d(X):=\sup\left\{ \rho>0:\, Q_{\rho}(X)\subset Q_{2}(0)\right\} $,
then $d(X)\geq1$ in $Q^{0}$. Hence, if we consider $Q^{1}:=B_{2}(0)\times(-3,-2)$
we conclude that
\begin{equation}
\sup_{Q^{0}}u\leq M:=\sup_{Q^{1}}d^{\gamma}u,\label{eq:77}
\end{equation}
where $\gamma$ is chosen at the same as the $\gamma$ in Lemma \ref{lem:quotient}
with $h=\frac{1}{2}$. From the discussion before Theorem \ref{thm:Lepsilon},
we can find $\exists X_{0}\in\overline{Q^{1}}\backslash\left[\partial_{p}Q^{1}\cap\partial_{p}Q_{2}\right]$
such that
\begin{equation}
d^{\gamma}u(X_{0})=M.\label{eq:78}
\end{equation}
Similarly as above, we define
\begin{equation}
\rho=\frac{1}{4}d(X_{0})\in(0,\frac{1}{2}],\label{eq:79}
\end{equation}
and
\begin{equation}
Q_{0}=Q_{\rho}(X_{0})\cap\left\{ u>\frac{1}{2}u(X_{0})\right\} .\label{eq:710}
\end{equation}
By the above discussion, we obtain
\[
\left|Q_{0}\right|>\mu_{1}\left|Q_{\rho}(X_{0})\right|
\]
for some constant $\mu_{1}=\mu_{1}(n,\nu,S,\gamma)>0$. Now we apply
the Corollary \ref{lower} with
\[
v=\frac{2}{u(X_{0})}u,\,\, Q_{r}=Q_{2\rho}(Y_{0}),\,\, Y_{0}=(x_{0},t_{0}+3\rho^{2}),\,\, Q^{0}=Q_{\rho}(X_{0}),\,\,1-\mu=\mu_{1}.
\]
Then we have
\begin{equation}
u\geq\beta u(X_{0})\label{eq:711}
\end{equation}
on $Q_{\rho}(Y_{0})$ with $\beta=\beta(n,\nu,S)>0$. Next we apply
Lemma \ref{lem:quotient} with
\[
v=u,\,\, r=2,\,\, D_{\rho}=B_{\rho}(x_{0})\times\{t_{0}+2\rho^{2}\}\subset\overline{Q_{\rho}(Y_{0})},
\]
and
\[
D^{0}=B_{1}(0)\times\{\tau\},\,\,\forall\tau\in(-1,0).
\]
So we have
\begin{equation}
\beta u(X_{0})\leq\inf_{D_{\rho}}u\leq\left(\frac{4}{\rho}\right)^{\gamma}\inf_{Q_{1}(0)}u.\label{eq:712}
\end{equation}
Finally, with the help of the intermediate region, we conclude that
\begin{equation}
\sup_{Q^{0}}u\leq M=d^{\gamma}u(X_{0})=\left(4\rho\right)^{\gamma}u(X_{0})\leq\beta^{-1}4^{2\gamma}\inf_{Q_{1}(0)}u.\label{eq:713}
\end{equation}
Taking $N=N(n,\nu,S)=\beta^{-1}4^{2\gamma}$ gives the desired result.
\end{proof}
It is well-known that it is easy to derive the H\"older continuity
of solutions from the Harnack inequality by standard oscillation and
iteration arguments.
\begin{thm}\label{hoelder}
Suppose $u\in W_{p}$ where $p<n+1$ is from Theorem \ref{variant}, and $u$ is a solution of $-u_{t}+Lu=0$ in $Q_{r}$. Then u is H\"older
continuous in $Q_{\frac{r}{2}}$.
\end{thm}

\section{Approximation}

In all the proofs from above sections, we always assume $u$ is $C^{2,1}$ in stead of $W_{p}$ where $p<n+1$ is from Theorem \ref{variant}.
In this section, we briefly show we can use an approximation argument
to show that all results hold for $u\in W_{p}(Q_{2r})$ in the sense of Definition \ref{space} and $p$ is from the variant Aleksandrov-Bakelman-Pucci-Krylov-Tso estimate \eqref{abpkt}.
Throughout, we assume
\begin{equation}
u\geq0,\,\,-u_{t}+Lu=-u_{t}+\sum_{ij}a_{ij}D_{ij}u+\sum_{i}b_{i}D_{i}u=0\label{eq:81}
\end{equation}
in $Q_{2r}$. We can approximate $a_{ij}$, $b_{i}$ and $u$ by smooth
functions $a_{ij}^{\epsilon}\rightarrow a_{ij}$, $b_{i}^{\epsilon}\rightarrow b_{i}$
a.e. as $\epsilon\rightarrow0^{+}$. And $u^{\epsilon}\rightarrow u$
in $W_{p}^{2,1}$ as $\epsilon\rightarrow0^{+}$. Then
\begin{equation}
f^{\epsilon}=-u_{t}^{\epsilon}+L^{\epsilon}u^{\epsilon}=-u_{t}^{\epsilon}+\sum_{ij}a_{ij}^{\epsilon}D_{ij}u^{\epsilon}+\sum_{i}b_{i}^{\epsilon}D_{i}u^{\epsilon}\rightarrow0\label{eq:82}
\end{equation}
in $L_{loc}^{p}(Q_{2r})$ as $\epsilon\rightarrow0^{+}$. We know
the existence of solutions to equations with smooth coefficients, therefore we can
write
\[
u^{\epsilon}=v^{\epsilon}+w^{\epsilon},
\]
where
\[
-v_{t}^{\epsilon}+L^{\epsilon}v^{\epsilon}=0
\]
 in $Q_{2r}$ and
\[
v^{\epsilon}=u^{\epsilon}
\]
 on $\partial_{p}Q_{2r}$:
\[
-w_{t}^{\epsilon}+L^{\epsilon}w^{\epsilon}=f^{\epsilon},
\]
\[
w^{\epsilon}=0
\]
 on $\partial_{p}Q_{2r}$. By the variant Aleksandrov-Bakelman-Pucci-Krylov-Tso
estimate \eqref{abpkt}, we know $w^{\epsilon}\rightarrow0$ in $L^{\infty}$ and
$v^{\epsilon}$ satisfies the Harnack inequality. Finally, by
an easy limiting argument, $u$ also satisfies the Harnack inequality.
\section{Appendix}
As we mentioned in the introduction, when drift $b\in L_{x,t}^{n+1}$, we do not expect the solution to have H\"older continuity since $ L_{x,t}^{n+1}$ is supercritical with respect to the parabolic scaling. In this appendix, we present a concrete example.
We consider the parabolic equation in $1+1$ dimensions
with drift $b\in L_{x,t}^{2}(\mathbb{R}^{2})$, 
\begin{equation}
u_{t}+b\nabla u-\Delta u=0.\label{eq:81}
\end{equation}
We define for $t\in[0,1]$ 
\begin{equation}
b(x,t)=a(t)\begin{cases}
1 & -r(t)\leq x<0\\
-1 & 0<x\leq r(t)\\
0 & x\notin[-r(t),\,0)\cup(0,\, r(t)]
\end{cases}\label{eq:82}
\end{equation}
and if $t\notin[0,1]$, $b=0$. We set $a(t)=(1-t)^{-\beta}$ and
$r(t)=(1-t)^{\alpha}$ where $\beta$ and $\alpha$ to be determined
later. First of all, by the integrability condition of $b$, we see
$\int_{0}^{1}(1-t)^{\alpha-2\beta}<\infty$, we get $\alpha-2\beta>-1$.

We try to construct an odd function $\phi$ so that we can do a  comparison argument.
For $0\leq x\leq1$, we define $\ensuremath{\phi(x)=\sin(\pi x/2)}$
and $\phi=1$ for $x>1$. Notice that for $x\in[0,1)$, we have $-\Delta\phi\leq C\phi$
for some constant $C$. In particular, based on our specific choice,
we take the constant $C=\left(\frac{2}{\pi}\right)^{2}$ . Finally,
we extend this $\phi$ oddly to the whole line.  We consider 
\begin{equation}
v(x,t)=\exp\left[-C\int_{0}^{t}(1-s)^{-2\alpha}\right]\phi\left(x/r(t)\right)\label{eq:83}
\end{equation}
which requires $-2\alpha>-1$. We try to verify on $\left(0,\, r(t)\right)$,
$v$ is a subsolution for $t\in[0,1]$ 
\begin{equation}
v_{t}=\exp\left[-C\int_{0}^{t}(1-s)^{-2\alpha}\right]\left(x\alpha(1-t)^{-a-1}\phi'-C(1-t)^{-2\alpha}\phi\right)\label{eq:84}
\end{equation}

\begin{equation}
b\nabla v=-(1-t)^{-\beta}(1-t)^{-\alpha}\phi'\exp\left[-C\int_{0}^{t}(1-s)^{-2\alpha}\right]\label{eq:85}
\end{equation}

\begin{equation}
-\Delta v=-(1-t)^{-2\alpha}\phi''\exp\left[-C\int_{0}^{t}(1-s)^{-2\alpha}\right]\label{eq:86}
\end{equation}
By construction, 
\begin{equation}
-(1-t)^{-2\alpha}\phi''\exp\left[-C\int_{0}^{t}(1-s)^{-2\alpha}\right]-C(1-t)^{-2\alpha}\phi\exp\left[-C\int_{0}^{t}(1-s)^{-2\alpha}\right]\leq0.\label{eq:87}
\end{equation}
We only need to verify that 
\begin{equation}
C\left(x\alpha(1-t)^{-a-1}\phi'\right)-(1-t)^{-\beta}(1-t)^{-\alpha}\phi'\exp\left[-C\int_{0}^{t}(1-s)^{-2\alpha}\right]\leq0\label{eq:88}
\end{equation}
for $x\in\left(0,r(t)\right)$ and $t\in[0,1)$. Since $\phi'$ is
nonnegative, it suffices to check 
\begin{equation}
x\alpha(1-t)^{-a-1}-(1-t)^{-\beta}(1-t)^{-\alpha}\leq0,\,\, x\in[0,r(t)),\,\, t\in[0,1).\label{eq:89}
\end{equation}
\begin{equation}
x\alpha(1-t)^{-a-1}\leq\alpha(1-t)^{-\alpha-1}(1-t)^{\alpha}=\alpha(1-t)^{-1}.\label{eq:810}
\end{equation}
We pick $-\beta-\alpha+1<0$. Also $-2\alpha>-1$ and $\alpha-2\beta>-1$.
So $\alpha\in(\frac{1}{3},\frac{1}{2})$, we can pick $\alpha=\frac{5}{12}$.
Then we pick $\beta=\frac{2}{3}$. This pair satisfies all of our
conditions. So with our choice of $\alpha$ and $\beta,$ we can see
our $v(x,t)$ is a subsolution to our equation in $[0,r(t))\times[0,1)$. 

By the symmetry of our equation and the oddness of $\phi$, $v$ is a
subsolution of our equation $[0,r(t))\times[0,1)$ and is a supersolution
on $(-r(t),0]\times[0,1)$. If we take $u$ to be the solution of our
equation with initial data $v(x,0)$. Then when $t$ approach $1$,
we look at the oscillation on the ball $B_{r(t)}(0)$. We have 
\begin{equation}
2\leq Osc_{B_{r(t)}(0)}v\leq Osc_{B_{r(t)}(0)}u.\label{eq:811}
\end{equation}
Since $r(t)\rightarrow0$ as $t\rightarrow1$, we conclude that $u$
will have a discontinuity at the origin when $t\rightarrow1$.

\section{Acknowledgment}
This work was initiated when I was still at the University of Minnesota Twin Cities. I would like to thank Professor Mikhail Safonov for suggesting this interesting problem to me, and for many motivating discussions. I also want to thank Professor Luis Silvestre for discussions on the supercritical scaling case and some related comments.

\end{document}